\RequirePackage{etex}
\RequirePackage{easymat}

\documentclass[12pt]{elsarticle}
\usepackage{amsthm,amsmath,amstext,amssymb,
extsizes}

\usepackage[all]{xy}
\usepackage[active]{srcltx}
 \newtheorem{theorem}{Theorem}
 \newtheorem{lemma}{Lemma}
\theoremstyle{definition}
\newtheorem{definition}{Definition}

\theoremstyle{remark}
\newtheorem{remark}{Remark}

\renewcommand{\le}{\leqslant}
\renewcommand{\ge}{\geqslant}

\DeclareMathOperator{\diag}{diag}

\newcommand{\ci}{
\begin{picture}(6,6)
\put(3,3){\circle*{3}}
\end{picture}}

\begin{document}

\title{Systems of subspaces of a unitary space}

\author[im]{Vitalij M. Bondarenko}
\ead{vit-bond@imath.kiev.ua}

\author[br]{Vyacheslav Futorny}
\ead{futorny@ime.usp.br}

\author[klim]{Tatiana Klimchuk}
\ead{klimchuk.tanya@gmail.com}

\author[im]{Vladimir V. Sergeichuk}
\ead{sergeich@imath.kiev.ua}

\author[br]{Kostyantyn Yusenko}
\ead{kay.math@gmail.com}

\address[im]{Institute of Mathematics, Tereschenkivska 3, Kiev, Ukraine.}

\address[br]{Department of Mathematics, University of S\~ao Paulo, Brazil.}

\address[klim]
{Faculty of Mechanics and Mathematics,
Kiev National Taras Shevchenko
University, Kiev, Ukraine.}

\begin{abstract}
For a finite poset ${\cal P}=\{p_1,\dots,p_t\}$, we study systems $(U_1,\dots,U_t)_U$ of subspaces $U_1,\dots,U_t$ of a unitary space $U$ such that $U_i\subseteq U_j$ if $p_i\prec p_j$.
Two systems
$(U_1,\dots,U_t)_U$ and $(V_1,\dots,V_t)_V$ are said to be isometric if there exists an isometry $\varphi:U\to V$ such that $\varphi(U_i)=V_i$. We classify such systems up to isometry if $\cal P$ is a semichain. We prove that the problem of their classification is unitarily wild if $\cal P$ is not a semichain. A classification problem is called unitarily wild if it contains the problem of classifying linear operators on a unitary space, which is hopeless in a certain sense.
\end{abstract}


\maketitle

\section{Introduction}

For a given poset, we consider its representations by systems of subspaces of a unitary space ordered by inclusion. We classify such systems for all posets for which an explicit classification is possible.

In this paper, we denote by $\mathcal P$ a finite set  with a partial order $\preccurlyeq$ whose elements $p_1,\dots,p_t$ are enumerated such that $p_i\prec p_j$ implies $i<j$.

\begin{definition}\label{def1}
A \emph{$\cal P$-system of subspaces of a unitary space} (\emph{$\cal P$-system} for short) is a system $(U_1,\dots,U_t)_U$ in which $U$ is  a unitary space and $U_1,\dots,U_t$ are  its subspaces such that $U_i\subseteq U_j$ if $p_i\prec p_j$. (Therefore, a $\cal P$-system is defined by
a homomorphism from $\cal P$ to the poset of all subspaces of a unitary space $U$.) Two $\cal P$-systems $(U_1,\dots,U_t)_U$ and $(V_1,\dots,V_t)_V$ are \emph{isometric} if there exists an isometry $\varphi: U\to V$ such that $\varphi (U_1)=V_1,\dots, \varphi (U_t)=V_t$. (Recall that a bijection $\varphi: U\to V$ is an \emph{isometry} if $(x,y)=(\varphi x,\varphi y)$ for all $x,y\in U$.) The problem is to classify $\cal P$-systems up to isometry.
\end{definition}

In particular, if all elements of $\cal P$ are incomparable, then we get the problem of classifying $t$-tuples of subspaces of a unitary space.
\begin{itemize}
  \item
\emph{Pairs of subspaces} of a unitary space and pairs of closed subspaces of a Hilbert space were studied and classified by many authors; see \cite{dixm,hal,gal}, the bibliography in \cite{gal}, and \cite[Section I.5]{ste}. Halmos \cite{hal} writes: ``Specialization   to  the  finite-dimensional   case
makes  neither  the  conclusions  more  obvious  nor  the  proofs  substantially   simpler''. Pairs of subspaces of a space with an indefinite scalar product over a field $\mathbb F$ of characteristic not $2$ were classified in \cite{ser_pairs}  up to classification of quadratic and Hermitian forms over finite extensions of $\mathbb F$.

  \item
The problem of classifying \emph{triples of subspaces} of a unitary space is unitarily wild; see \cite{krug_80,Serg_un}. A classification problem is called \emph{unitarily wild} if it contains the problem of classifying linear operators on unitary spaces. The latter problem contains the problem of classifying \emph{any} system of linear mappings on unitary spaces; see \cite{krug_80} and \cite[Section 2.3]{Serg_un}. Thus, all unitarily wild problems of classifying systems of mappings on unitary spaces have the same complexity and a solution of one would imply a solution of each other. By this reason, we cannot expect to get an observable solution to any unitarily wild problem.
\end{itemize}

\begin{definition}\label{def2}
A poset
${\cal P}=\{p_1,\dots,p_t\}$ is a \emph{chain} if  $p_1 \prec p_2 \prec \dots \prec p_t$.
A poset $\cal P$ is a \emph{semichain} if it has the form
\begin{equation}\label{kjy}
\mathcal P_1 \prec \mathcal P_2 \prec \dots \prec \mathcal P_s
\end{equation}
in which every $\mathcal P_i$ consists of one or two incomparable elements and $\mathcal P_i\prec\mathcal P_{i+1}$ means that $a\prec b$ for all $a\in\mathcal P_{i}$ and $b\in\mathcal P_{i+1}$.
\end{definition}
For example, a poset with Hasse diagram
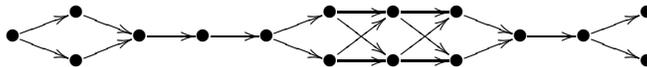
\[
\xymatrix@=2pt{
&&&*{\bullet}\ar@{<-}[llld]\ar[rrrd]
 &&&&&&&&&&&&*{\bullet}\ar@{<-}[llld]\ar[rrr]
 \ar[rrrdd]
 &&&*{\bullet}\ar[rrr]\ar[rrrdd]&&&
 *{\bullet}\ar[rrrd]
&&&&&&&&&*{\bullet}\ar@{<-}[llld]
\\
*{\bullet}&&&&&&*{\bullet}\ar[rrr]&&&
*{\bullet}
\ar[rrr]&&&*{\bullet}&&&&&& &&&&&&
*{\bullet}\ar[rrr]&&&*{\bullet}
\\&&&*{\bullet}\ar@{<-}[lllu]\ar[rrru]
&&&&&&&&&&&&*{\bullet}\ar@{<-}[lllu]\ar[rrr]
\ar[rrruu]
&&&*{\bullet}\ar[rrr]\ar[rrruu]&&&*{\bullet}\ar[rrru]
&&&&&&&&&*{\bullet}\ar@{<-}[lllu]
}
\]
($a\longrightarrow b$ denotes $a\prec b$) is a semichain. Semichains are often appear in representation theory; see \cite{bon}.

We prove that \emph{the problem of classifying $\cal P$-systems up to isometry is unitarily wild if and only if $\cal P$ is not a semichain and classify $\cal P$-systems up to isometry for each semichain $\cal P$.}

Note that the problem of classifying systems of subspaces in \emph{vector spaces} (without scalar product) is much more meaningful; see Section \ref{kudr}. In particular, the problem of classifying $t$-tuples of subspaces of a vector space is trivial for $t=2$, it is not difficult for $t=3$, it was solved by Gelfand and Ponomarev \cite{gel_pon_chetverki} (see also \cite{bren,med-zav}) for $t=4$, and it is hopeless if $t\ge 5$ (see (b) in Section \ref{kudr}).

The paper is organized as follows. In Section \ref{s2} we formulate two main theorems. In Section \ref{sss3} we reformulate them in the matrix form.  In Sections \ref{kpkd} and \ref{sskw} we prove the main theorems. In Section \ref{kudr} we compare them with classical results about systems of subspaces of a vector space. In Section \ref{appen} we explain the origin of the integral quadratic form \eqref{khtr}.

\section{Two main theorems}\label{s2}

The \emph{orthogonal direct sum} of two $\cal P$-systems $\mathcal U=(U_1,\dots,U_t)_U$ and $\mathcal V=(V_1,\dots,V_t)_V$ is the $\cal P$-system \begin{equation}\label{ouo}
\mathcal U\perp \mathcal V:=(U_1+ V_1,\dots,U_t+ V_t)_{U\perp V},
\end{equation}
in which $U\perp V$ is the orthogonal direct sum of unitary spaces $U$ and $V$.

A $\cal P$-system $\mathcal U=(U_1,\dots,U_t)_U$ is \emph{indecomposable} if $U\ne 0$ and $\mathcal U$ is not isometric to an orthogonal direct sum of $\cal P$-systems of subspaces of unitary spaces of smaller dimensions.

\begin{definition}\label{defw}
\begin{itemize}
  \item
A poset $\cal P$ is \emph{unitarily representation-finite} if it has only a finite number of nonisometric $\cal P$-systems that are unitarily indecomposable.

    \item
A poset $\cal P$ is \emph{unitarily wild} if the problem of classifying $\cal P$-systems up to isometry contains the problem of classifying operators on unitary spaces. The posets that are not unitarily wild are called \emph{unitarily tame} (in analogy with
the partition of animals into
wild and tame ones).
\end{itemize}
\end{definition}
Clearly, each unitarily representation-finite poset is unitarily tame.

For a poset ${\cal P}=\{p_1,\dots,p_t\}$, we define the integral quadratic form
\begin{equation}\label{khtr}
\begin{split}
u_{\mathcal P}(x_0,&x_1,\dots,x_t)\\&:=x_0^2+2\Big(x_1^2+\dots+x_t^2 + \sum_{p_i\prec p_j} x_ix_j - x_0(x_1+\dots+x_t)\Big)
\end{split}
\end{equation}
in which the sum is taken over all pairs of elements $p_i,p_j\in\cal P$ satisfying $p_i\prec p_j$. This form can be called the \emph{Tits form for unitary representations of the poset ${\cal P}$} since in plays the same role and is constructed in the same way (see Section \ref{appen}) as the Tits form \eqref{kde1} of a poset and the Tits form of a quiver (see \cite[Section 7.1]{gab_roi} or \cite[Section 2.4]{haz}).

\begin{theorem}[proved in Section \ref{sskw}]\label{ttt1}
\begin{itemize}
  \item[\rm(a)] The following three conditions are equivalent for a finite poset $\mathcal P$:
	\begin{itemize}
	    \item[{\rm(i)}] $\mathcal P$ is unitarily representation-finite,
		
\item[{\rm(ii)}] $\mathcal P$ is a chain,	

\item[{\rm(iii)}] the form $u_{\mathcal P}$ is positive definite.	

	\end{itemize}	

  \item[\rm(b)] The following three conditions are equivalent for a finite poset $\mathcal P$:
	\begin{itemize}
	    \item[{\rm(i)}] $\mathcal P$ is unitarily tame,
		
\item[{\rm(ii)}] $\mathcal P$ is a semichain,	

\item[{\rm(iii)}] the form $u_{\mathcal P}$ is  nonnegative definite.
	
	\end{itemize}	
\end{itemize}
\end{theorem}

Recall that an integral quadratic form $q:\mathbb Z^n\rightarrow \mathbb Z$ is called
\begin{itemize}
  \item \textit{positive definite} if $q(z)> 0$
  \item \textit{nonnegative definite} if $q(z)\ge 0$
\end{itemize}
for all nonzero $z=(z_1,\dots,z_n)\in\mathbb Z^n$. It
is called
\begin{itemize}
  \item \textit{weakly positive definite} if $q(z)> 0$
  \item \textit{weakly nonnegative definite} if $q(z)\ge 0$
\end{itemize}
for all nonzero $ z=(z_1,\dots,z_n)\in\mathbb Z^n$ with nonnegative $z_1,\dots,z_n$.

\begin{remark}\label{re1}
We prove in Theorem \ref{th4}(ii) that for each finite poset $\mathcal P$ the form $u_{\mathcal P}$
is positive definite if and only if it is weakly positive definite; $u_{\mathcal P}$ is nonnegative definite if and only if it is weakly nonnegative definite. The same holds for the Tits form of a quiver, but does not hold for the Tits form of the poset \eqref{lche}, and so positive and nonnegative definiteness cannot be used in the nonunitary analogue of Theorem \ref{ttt1} (see the statements (a) and (b) in Section \ref{kudr}).
\end{remark}

Define the following indecomposable $\mathcal P$-systems for a semichain $\mathcal P=\{p_1,\dots,p_t\}$:
\begin{description}
  \item[${\cal F}_k$]$\!\!:=\bigl(\,\underbrace{0,\dots,0}_{k-1}\,,
\mathbb C,\dots,\mathbb C\bigr)_{\mathbb C}$ for each $k=1,\dots,t$,

  \item[${\cal F}$]$\!\!:=(0,\dots,0)_{\mathbb C}$,

  \item[${\cal G}_{k,\sigma }$]$\!\!:=\bigl(\,\underbrace{0,\dots,0}_{k-1}\,,\mathbb C(1,0),\mathbb C(\sigma ,1),\mathbb C\perp \mathbb C,\dots,\mathbb C\perp \mathbb C\bigr)_{\mathbb C\perp \mathbb C}$
for each $p_k\nprec p_{k+1}$ and for each positive real $\sigma  $ (here $(1,0)$ and $(\sigma ,1)$ are the elements of $\mathbb C\perp \mathbb C$),

  \item[${\cal G}_{k}$]$\!\!:=\bigl(\,\underbrace{0,\dots,0}_{k-1}\,,
\mathbb C,0,\mathbb C,\dots,\mathbb C\bigr)_{\mathbb C}$
for each $p_k\nprec p_{k+1}$.
\end{description}

For each unitarily tame poset $\mathcal P$, $\mathcal P$-systems are classified up to isometry in the following theorem.

\begin{theorem}[proved in Section \ref{kpkd}]\label{ttt2}
\begin{itemize}
  \item[\rm(a)] If $\mathcal P$ is a chain, then each $\mathcal P$-system of subspaces of a unitary space is isometric to an orthogonal direct sum, uniquely determined up to permutation of summands, of $\mathcal P$-systems of the form $\mathcal F_1,\dots,\mathcal F_{t},\mathcal F$.

  \item[\rm(b)]  If $\mathcal P$ is a semichain, then each $\mathcal P$-system of subspaces of a unitary space is isometric to an orthogonal direct sum, uniquely  determined up to permutation of summands, of $\mathcal P$-systems of the form $\mathcal F_1,\dots,\mathcal F_{t},\mathcal F$ and of the form
$\mathcal G_{k,\sigma},\mathcal G_{k}$ in which
$p_k\nprec p_{k+1}$ and $\sigma $ is a positive real number.	
\end{itemize}
\end{theorem}

\section{The matrix form of the main theorems}\label{sss3}

We suggest that there are
matrices $0_{n0}$ and $0_{0n}$ of sizes $n\times 0$ and $0\times n$ for every nonnegative integer $n$; they represent the linear
mappings $0\to {\mathbb C}^n$ and
${\mathbb C}^n\to 0$.

A $\mathcal P$-system ${\cal U}=(U_1,\dots,U_t)_U$ can be given in an orthonormal basis $e_1,\dots,e_m$ of $U$ by a block matrix
\begin{equation*}\label{lwe}
A_{\cal U}=[A_1|\dots|A_t]
\end{equation*}
in which each block $A_i$ is constructed as follows:
\begin{equation}
\label{mal}
\parbox{25em}
{choose a subspace $V_i$ of $U$ such that $V_i+\sum_{p_j\prec p_i}U_j=U_i$ (in particular, $V_1=U_1$), then the columns of $A_i$ are the coordinate vectors of any system of vectors spanning $V_i$ in the basis $e_1,\dots,e_m$.
}
\end{equation}

Conversely, each block matrix $A=[A_1|\dots|A_t]$ with $m$ rows defines the $\mathcal P$-system
\begin{equation}\label{bbr}
f(A):=(U_1,\dots,U_t)_{\mathbb C^m}
\end{equation}
of subspaces of the unitary space $\mathbb C^m$ with the usual scalar product, in which $U_i$ is  spanned by the columns of all $A_i$ such that $p_i\preccurlyeq p_j$.

The block matrix $A_{\cal U}$ is determined by ${\cal U}$ up to a weak unitary $\cal P$-equivalence, which is defined as follows.

\begin{definition}\label{def3}
\begin{itemize}
  \item[(a)]
Two block matrices
\begin{equation}\label{qsy}
A=[A_1|\dots|A_t],\qquad B=[B_1|\dots|B_t]
\end{equation}
are \emph{unitarily $\cal P$-equivalent} if $B$ can be obtained from $A$ by a sequence of the following transformations:
\begin{itemize}
  \item[(i)]
arbitrary unitary transformations of rows,

  \item[(ii)]
arbitrary elementary transformations of columns within each vertical strip,

  \item[(iii)]
additions of linear combinations of columns of strip $i$ to columns of strip $j$ if $p_i\prec p_j$.
\end{itemize}

  \item[(b)]
Two block matrices \eqref{qsy} are \emph{weakly unitarily $\cal P$-equivalent} if one can adjoin zero columns to some of their blocks and obtain unitarily $\cal P$-equivalent block matrices.
\end{itemize}
\end{definition}

We say that a block matrix $A=[A_1|\dots|A_t]$ is of \emph{size} $m\times (n_1,\dots,n_t)$ if each $A_i$ is $m\times n_i$. Unitarily $\cal P$-equivalent matrices have the same size.

It is easy to see that two block matrices $A$ and $B$ of size $m\times (n_1,\dots,n_t)$ are unitarily $\cal P$-equivalent if and only if $B$ can be obtained from $A$ by transformations
\begin{equation}\label{dru}
A\mapsto RAS,\qquad
\begin{array}{c}
\text{$R$ is unitary, $S=[S_{ij}]_{i,j=1}^t$ is nonsingular} \\ \text{and upper block-triangular,
$S_{ij}$ is $n_i\times n_j$,} \\
\text{and $S_{ij}=0$ if $p_i\not\prec p_j.$}
\end{array}
\end{equation}

Due to the following lemma, the problem of classifying $\mathcal P$-systems up to isometry is reduced to the problem of classifying block matrices up to weak unitary $\cal P$-equivalence.

\begin{lemma}\label{lem1}
\begin{itemize}
  \item[\rm(a)] For each
$\mathcal P$-system ${\cal U}$, there exists a block matrix $A=[A_1|\dots|A_t]$ such that $f(A)$ defined in \eqref{bbr} is isometric to ${\cal U}$.

  \item[\rm(b)] Two block matrices $A$ and $B$ are weakly unitarily $\cal P$-equivalent if and only if $f(A)$ and $f(B)$ are isometric.
\end{itemize}
\end{lemma}

\begin{proof}
(a) If ${\cal U}=(U_1,\dots,U_t)_U$ is a $\mathcal P$-system and $A_{\cal U}=[A_1|\dots|A_t]$ is constructed by \eqref{mal}, then $f(A_{\cal U})$ is isometric to ${\cal U}$ because
$A_{\cal U}$ can be also constructed by induction as follows:
$A_1$ is a matrix whose columns are the coordinate vectors $[a_{1}]_e,\dots,[a_{n_1}]_e$ in some orthonormal basis $e_1,\dots,e_m$ of $U$ of any system of vectors $a_{1},\dots,a_{n_1}$ spanning $U_1$. Let $A_1,\dots,A_k$ ($k<t$) have been constructed. Then $A_{k+1}$ is an arbitrary matrix such that its columns and the columns of all $A_i$ with $p_i\prec p_{k+1}$ are the coordinate vectors of a system of vectors spanning $U_{k+1}$.

(b) ``$\Rightarrow$'' Let $A$ and $B$ be weakly unitarily $\cal P$-equivalent. Then one can adjoin zero columns to some of their blocks and obtain unitarily $\cal P$-equivalent block matrices $\tilde A$ and $\tilde B$.
There exist $R$ (let its size be $m\times m$) and $S$ satisfying \eqref{dru} such that $R\tilde AS=\tilde B$.
Then $f(A)$ and $f(B)$ are isometric via the isometry \[\varphi: \mathbb C^m\to \mathbb C^m,\qquad v\mapsto Rv.\]

``$\Leftarrow$'' Let
$f(A)$ and $f(B)$ be isometric via an isometry $\varphi: \mathbb C^m\to \mathbb C^m$ given by an $m\times m$ matrix $R$; i.e., $\varphi(v)=Rv$ for all $v\in\mathbb C^m$. We need to construct
block matrices $\tilde A$ and $\tilde B$ (adjoining zero columns to some blocks of $A$ and $B$) and a matrix $S$ satisfying \eqref{dru} so that $R\tilde AS=\tilde B$. Replacing $A$ and $\tilde A$ by $RA$ and $R\tilde A$, we reduce our consideration to the case $R=I$; that is, to the case
\begin{equation}\label{wut}
f(A)=f(B).
\end{equation}

We use induction on $t$.

If $t=1$, then we adjoin zero columns to $A$ or $B$ so that the obtaining matrices $\tilde A$ and $\tilde B$ have the same number of columns. By \eqref{wut}, the columns of $\tilde A$ and $\tilde B$ span the same vector space; let $v_1,\dots,v_r$ be its basis. Since both $\tilde A$ and $\tilde B$ can be reduced to the matrix $[v_1\dots v_r0\dots 0]$ by column-transformations, there exists a nonsingular $S$ such that $\tilde AS=\tilde B$.

Denote by $A'$ and $B'$ the block matrices $A$ and $B$ without the last blocks $A_t$ and $B_t$. Denote by $\mathcal P'$ the poset $\mathcal P$ without the element $p_t$. By \eqref{wut}, $A'$ and $B'$ determine the same $\mathcal P'$-system $f(A')=f(B')$. By induction hypothesis, we can construct block matrices $\tilde A'$ and $\tilde B'$ adjoining zero columns to some blocks of $A'$ and $B'$ and a matrix $S'$ satisfying \eqref{dru} so that $\tilde A'S'=\tilde B'$. Replacing the submatrices $A'$ and $B'$ of $A$ and $B$ with $\tilde A'S'$ and $\tilde B'$, we make $A'=B'$. Adjoin zero columns to $A$ or $B$ to the right so that the obtaining matrices $\tilde A$ and $\tilde B$ have the same number of columns. By \eqref{wut}, $f(\tilde A)=f(\tilde B)=(U_1,\dots,U_t)_{\mathbb C^m}$, in which  $U_t$ is spanned by columns of all $A_i$ such that $p_i\preccurlyeq p_t$. Let $v_1,\dots,v_r$ be a system linearly independent vectors such that it and the columns of all $A_i(=B_i)$ with $p_i\prec p_t$ span $U_t$. Since both $\tilde A$ and $\tilde B$ can be reduced to the block matrix $[A_1|\dots|A_{t-1}|v_1\dots v_r0\dots 0]$ by column-transformations from Definition \ref{def3}(a), there exists a nonsingular $S$ satisfying \eqref{dru} such that $\tilde AS=\tilde B$.\end{proof}

The orthogonal direct sum $\mathcal U\perp \mathcal V$ of $\cal P$-systems $\mathcal U=(U_1,\dots,U_t)_U$ and $\mathcal V=(V_1,\dots,V_t)_V$ defined in  \eqref{ouo} corresponds to the \emph{block direct sum}
\begin{equation}\label{grp}
	A_{\cal U}\boxplus A_{\cal V}:=
\left[\begin{array}{cc|cc|c|cc}
A_1&0&A_2&0&\ldots&A_t&0\\0& B_1&0&B_2& \ldots&0&B_t\end{array}\right]
\end{equation}
of block matrices $A_{\cal U}=[A_1|\dots|A_t]$ and
$A_{\cal V}=[B_1|\dots|B_t]$.
A block matrix is \emph{unitarily indecomposable} if its size is not $0\times (0,\dots,0)$ and it is not unitarily $\cal P$-equivalent to a block direct sum of block matrices of smaller sizes.

Let us reformulate Theorem \ref{ttt2} in the matrix form.

Define the following unitarily indecomposable block matrices for a semichain $\mathcal P$:
\begin{description}
  \item[$E_k$]$\!\!
:=[0_{10}|\dots|0_{10}|1|0_{10}|\dots|0_{10}]
$
with $1$ in the $k$th strip, $k=1,\dots,t$;

  \item[$F$]$\!\!:=[0_{10}|\dots|0_{10}]
$;

  \item[$G_{k,\sigma}$]$\!\!:=\left[ \begin{array}{c|c|c|c|c|c|c|c}
0_{20}&\dots&0_{20}&
\begin{matrix}1\\0\end{matrix}
&
\begin{matrix}\sigma \\1\end{matrix}
&0_{20}&\dots&0_{20}
\end{array}\right]$ with $\begin{bmatrix}1\\0\end{bmatrix}$ in the $k$th strip, for each $p_k\nprec p_{k+1}$ and for each positive real $\sigma $;

  \item[$H_{k}$]$\!\!:=[0_{10}|\dots|0_{10}|1|1|0_{10}|
\dots|0_{10}]
$ with 1 in the $k$th and $(k+1)$st strips, for each $p_k\nprec p_{k+1}$;

\item[$L_k$]$\!\!:
=[0_{00}|\dots|0_{00}|0_{01}|0_{00}|\dots|0_{00}]
$
with $0_{01}$ in the $k$th strip, $k=1,\dots,t$.
\end{description}

The following matrix form of Theorem \ref{ttt2} will be proved in Section \ref{kpkd}.

\begin{theorem}\label{lemfr}
\begin{itemize}
  \item[\rm(a)] Let $\mathcal P$ be a chain. Each block matrix $A=[A_1|\dots|A_t]$ is unitarily $\cal P$-equivalent to a block direct sum of block matrices of the form $E_1,\dots,E_t,F,L_1,\dots,L_t$.
This block direct sum is uniquely determined, up to permutation of summands.

  \item[\rm(b)]  Let $\mathcal P$ be a semichain. Each block matrix $A=[A_1|\dots|A_t]$  is unitarily $\cal P$-equivalent to a block direct sum of block matrices of the form $E_1,\dots,E_t,F,L_1,\dots,L_t$, and $G_{k,\sigma },H_k$ in which
$p_k\nprec p_{k+1}$ and $\sigma $ is a positive real number.
This block direct sum is uniquely determined, up to permutation of summands.
\end{itemize}
\end{theorem}

\section{Proof of Theorems \ref{ttt2} and \ref{lemfr}}
\label{kpkd}

\begin{lemma}\label{lems}
Theorem \ref{ttt2} follows from Theorem \ref{lemfr}.
\end{lemma}

\begin{proof} This statement follows from Lemma \ref{lem1} since
\begin{itemize}
  \item[(a)] if $\mathcal P$ is a chain and $C$ is one of the block matrices $E_1,\dots,E_t,F$, then $f(C)$ is one of the $\cal P$-systems $\mathcal F_1,\dots,\mathcal F_t,\mathcal F$, respectively;
  \item[(b)] if $\mathcal P$ is a semichain, then $f(E_k)$ is $\mathcal F_k$ if $p_k\prec p_{k+1}$ or $\mathcal G_k$ if $p_k\nprec p_{k+1}$, $f(F)=\mathcal F$, $f(G_{k,\sigma })=\mathcal G_{k,\sigma }$, and $f(H_k)=\mathcal F_k$ with $p_k\nprec p_{k+1}$.
\end{itemize}
\end{proof}

\begin{proof}[Proof of Theorem \ref{lemfr}]
(a) Let $\mathcal P$ be a chain. Let us prove that each block matrix $A$ is unitarily $\cal P$-equivalent to exactly one block matrix of the form
\begin{equation}\label{aky}
\left[ \begin{array}{cc|cc|c|cc}
I&0&0&0&\dots&0&0\\0&0&I&0&\dots&0&0\\ \dots&\dots&\dots&\dots&\dots&\dots&\dots\\	 0&0&0&0&\dots&I&0\\0&0&0&0&\dots&0&0
\end{array}\right].
\end{equation}

Each complex matrix $M$ possesses a singular value decomposition \cite[Theorem 7.3.5]{hor}:
\begin{equation}\label{llv}
M=U\Sigma_M V,\qquad U,V\text{ are unitary, }\ \Sigma_M=\diag(\sigma_1,\dots,\sigma_r)\oplus 0,
\end{equation}
in which $\sigma_1\ge\dots\ge\sigma_r>0$ are the positive square roots of the nonzero eigenvalues of $MM^*$ and hence $\Sigma_M$ is uniquely determined.

We reduce $A=[A_1|\dots|A_t]$ to the form \eqref{aky} by transformations (i)--(iii) from Definition \ref{def3}. First reduce $A_1$ to the form $\Sigma_{A_1}$ using (i) and (ii), then to the form $I_r\oplus 0$ using (ii) and transform $A$ to the form
\begin{equation}\label{kec}
\left[ \begin{array}{cc|c|c|cccc}
I_r&0&0&\dots&0\\0&0&A_2'&\dots&A_t'
\end{array}\right]
\end{equation}
using (iii).
Reduce $[A_2'|\dots|A_t']$ analogously, and so on until obtain \eqref{aky}.

The matrix \eqref{aky} is uniquely determined by $A$, which is proved by induction:
\begin{itemize}
  \item $r$ is the rank of $A_1$;
  \item
it is straightforward to check that two block matrices $[A_1|\dots|A_t]$ and $[B_1|\dots|B_t]$ are unitarily $\cal P$-equivalent if and only if
$[A_2'|\dots|A_t']$ and $[B_2'|\dots|B_t']$ are unitarily $\mathcal P'$-equivalent, in which $\mathcal P'$ is the chain $\mathcal P$ without the first element.

\end{itemize}
Each block matrix \eqref{aky} is a block direct sum, uniquely  determined up to permutation of summands, of matrices of the form $E_1,\dots,E_t,F,L_1,\dots,L_t$.

(b) Let $\mathcal P$ be a semichain. If $p_1\prec p_{2}$, then we reduce $A$ to the form \eqref{kec}, which is a block direct sum of matrices of the form $E_1$ and $L_1$ and the matrix $[0_{m'0}|A_2'|\dots|A_t']$.
The summands of the form $E_1,L_1$ are determined by $A$ uniquely up to permutation. The matrix $A':=[A_2'|\dots|A_t']$ is uniquely determined up to ${\cal P}'$-equivalence, in which $\mathcal P'$ is the semichain $\mathcal P$ without the first element.
Reasoning by induction on $t$, we assume that the statement (b) holds for $A'$. Then it holds for $A$ (each summand of $A'$ gives the summand of $A$ with the empty first strip).

Let $p_1\nprec p_{2}$. We reduce $A$ by transformations (i)--(iii) from Definition \ref{def3}; i.e., by transformations $A\mapsto RAS$ defined in  \eqref{dru}.
First we make $A_1=I_r\oplus 0$ using (i) and (ii) and then reduce the other strips of $A$ by those transformations \eqref{dru} that preserve $A_1$ (i.e., $R$ and $S$ must satisfy $RA_1S_{11}=A_1$). Since $R$ is unitary, it has the form
\begin{equation}\label{jop}
 R=R_1\oplus R_2,\qquad R_1 \text{ is $r$-by-$r$.}
\end{equation}

Therefore, the matrix \[A_2=
\begin{bmatrix}
 A_{21} \\
A_{22}  \\
\end{bmatrix},\qquad \text{$A_{21}$ is $r$-by-$ n_2$,}\]
is reduced by unitary row-transformations within horizontal strips
and elementary column-transformations.
We reduce $A$ to the form
\begin{equation}\label{gtw}
\begin{MAT}(e){|ccc|c.cc|c|c|c|}
\first-
I&0&0&0&I_s&0&   0&\dots&0\\
0&I&0&\Sigma &0&0&   0&\dots&0\\.
0&0&0&I_l&0&0&   0&\dots&0\\
0&0&0&0&0&0&   A_3'&\dots&A_t'
\addpath{(3,3,.)r}\\-
\end{MAT}\,, \quad
\begin{matrix}
\Sigma=\diag(\sigma_1,\dots,\sigma_q)\oplus 0\\ \sigma_1\ge\dots\ge\sigma_q>0
\end{matrix}
\end{equation}
as follows. First we reduce $A_{22}$ to the form $I_l\oplus 0$ and partition $A_{21}$ into vertical strips
$[\begin{MAT}(@){c.c}
B_{1}&B_{2}\\
\end{MAT}]
$ prolonging the partition of $A_{22}$. Then reduce $B_{2}$ to the form $I_s\oplus 0$ and partition $B_{1}$ into horizontal strips prolonging the partition of $B_{2}$:
\[
A_{21}=[\begin{MAT}(@){c.c}
B_{1}&B_{2}\\
\end{MAT}]=\left[\begin{MAT}(@){c.cc}
B_{11}&I_s&0\\
B_{21}&0&0
\addpath{(0,1,.)r}\\
\end{MAT}\right]
\]
Make $B_{11}=0$ by adding linear combinations of columns of $I_s$.

We can reduce $B_{21}$ by unitary row-transformations. We can also reduce $B_{21}$ by unitary column-transformations since elementary column-transformations with $B_{21}$ spoil $I_l$, which can be restored only by unitary row-transformations. Therefore, $B_{21}$ is reduced to the form $\Sigma=\Sigma_{B_{21}}$ defined in \eqref{llv}.

At last, we make zero all entries in $A_3,\dots,A_t$ to the right of $I_r$ in $A_1$ and to the right of $I_l$ in $A_2$ and obtain \eqref{gtw}.

All blocks of \eqref{gtw} are uniquely determined by $A$, except for $A_3',\dots,A_t'$. The matrix $[A_3'|\dots|A_t']$ is uniquely determined, up to unitary $\mathcal P''$-equivalence, in which $\mathcal P''$ is the semichain obtained from $\mathcal P$ by deleting the first two elements.

The matrix \eqref{gtw} is a block direct sum of matrices of the form $E_1$, $E_2$, $G_{1,\sigma }$, $H_1,$ $L_1,$ $L_2$, and the matrix $[0_{m'0}|0_{m'0}|A_3'|\dots|A_t']$.
The summands of the form $E_1$, $E_2$, $G_{1,\sigma }$, $H_1,$ $L_1,$ $L_2$ are determined by $A$ uniquely, up to permutation. The matrix $A':=[A_3'|\dots|A_t']$ is uniquely determined, up to unitary ${\cal P}''$-equivalence.

Using induction on $t$, we can assume that the statement (b) holds for $A'$. Attaching two empty vertical blocks to the left of each block direct summand of $A'$, we obtain a block direct sum that is unitarily $\mathcal P$-equivalent to $[0_{m'0}|0_{m'0}|A_3'|\dots|A_t']$ and is uniquely determined up to permutation.
\end{proof}

\section{Proof of Theorem \ref{ttt1}} \label{sskw}

\subsection{Proof of the equivalence {\rm(i)}$\Leftrightarrow${\rm(ii)}} \label{sskw1}

(a) If $\mathcal P$ is a chain, then it is unitarily representation-finite by Theorem \ref{ttt2}(a). If $\mathcal P$ is not a chain, then it contains two points $p_k$ and $p_{k+1}$ that are not comparable. The block matrices $G_{k,\sigma}$ from Theorem \ref{lemfr}(b) are not unitarily $\cal P$-equivalent for distinct $\sigma$. By Lemma \ref{lem1}(b), $\mathcal P$ is not unitarily representation-finite.
		
(b) If $\mathcal P$ is a semichain, then it is unitarily tame by Theorem \ref{ttt2}(b). If $\mathcal P$ is not a semichain, then it contains three elements $p_i,p_j,p_k$  ($i<k$) such that $p_j$ is not comparable with $p_i$ and $p_k$; i.e., the Hasse diagram of the subset $\{p_i,p_j,p_k\}$ is
\begin{equation}\label{lce}
\begin{split}
\xymatrix@=10pt{
{p_k}\\
{p_i}\ar@{-}[u]&{p_j}&\qquad\text{or}
\qquad&{p_i}&{p_j}
&{p_k}}
\end{split}
\end{equation}

For each $4n$-by-$2n$ matrix $M$, define the block matrix $A(M)=[A_1|\dots|A_t]$ in which
\[
[A_i|A_j|A_k]:=\left[
 \begin{array}{c|c|c}
I_{4n}&\Sigma &0\\0&I_{4n}&M
\end{array}\right],\quad \Sigma:=I_n\oplus 2I_n\oplus 3I_n\oplus 4I_n,
\]
and the other strips are empty: $A_l:=0_{8n,0}$ if $l\ne i,j,k$. By \eqref{dru}, if $A(M)$ and $A(N)$ are unitarily $\cal P$-equivalent, then there exist a unitary matrix $R$ and a nonsingular matrix
\[
S=\begin{bmatrix}
    S_1 & 0&S_{13} \\
    0&S_2&0 \\0&0&S_3
  \end{bmatrix},\quad S_{13}=0\text{ if } p_i\nprec p_k,
\]
such that
\begin{equation}\label{kde}
R\left[
 \begin{array}{c|c|c}
I_{4n}&\Sigma &0\\0&I_{4n}&M
\end{array}\right]S=\left[
 \begin{array}{c|c|c}
I_{4n}&\Sigma &0\\0&I_{4n}&N
\end{array}\right].
\end{equation}
 By analogy with \eqref{jop}, $R=R_1\oplus R_2$, in which $R_1$ and $R_2$ are $4n\times 4n$ unitary matrices and $R_2=S_2^{-1}$. Therefore, $R_1\Sigma R_2^{-1}=\Sigma$, and so $R_1=U_1\oplus U_2\oplus U_3\oplus U_4$ in which $U_1,\dots, U_4$ are $n$-by-$n$.

Take
\begin{equation}\label{efy}
M:=\begin{bmatrix}
I&0\\0&I\\I&I\\I&X
   \end{bmatrix},\qquad
N:=\begin{bmatrix}
I&0\\0&I\\I&I\\I&Y
   \end{bmatrix}
\end{equation}
in which all blocks are $n\times n$ and $X,Y$ are arbitrary $n\times n$ matrices.

The equality $R_2MS_3=N$ falls into 4 equalities:
\[
U_1[I\ 0]S_3=[I\ 0],\ \ U_2[0\ I]S_3=[0\ I],\ \ U_3[I\ I]S_3=[I\ I],\ \ U_4[I\ X]S_3=[I\ Y].
\]
By the first equality, $S_3$ is lower block-triangular. By the second equality, $S_3$ is upper block-triangular. Hence, $S_3=U_1^{-1}\oplus U_2^{-1}$.
By the third equality, $U_1=U_2$.
By the fourth equality, $U_1=U_2=U_4$ and $U_4XU_4^{-1}=Y$. Therefore, $X$ and $Y$ are unitarily similar.

Conversely, if $X$ and $Y$ are unitarily similar; that is, $VXV^{-1}=Y$ for some unitary $V$, then \eqref{kde} holds for
\[
R=\underbrace{V\oplus\dots\oplus V}_{8 \text{ summands}}\ ,\qquad
S=\underbrace{V^{-1}\oplus\dots\oplus V^{-1}}_{10 \text{ summands}}\ ,
\]
and so $A(M)$ and $A(N)$ are  unitarily $\cal P$-equivalent. Note that $A(M)$ and $A(N)$ with $M$ and $N$ of the form \eqref{efy} are  unitarily $\cal P$-equivalent if and only if they are weakly unitarily $\cal P$-equivalent since they have no summands $[0_{00}|\dots|0_{00}|0_{01}|
0_{00}|\dots|0_{00}]$.

Therefore, the problem of classifying
block matrices up to weakly unitary $\cal P$-equivalence contains the problem of classifying square matrices up to unitary similarity. Lemma \ref{lem1}(b) ensures that $\mathcal P$ is unitarily wild.

\subsection{Proof of the equivalence {\rm(ii)}$\Leftrightarrow${\rm(iii)}} \label{sskw3}

(ii)$\Rightarrow$(iii). Let $\cal P$ be a chain. By \eqref{khtr},
\[
u_{\mathcal P}(x_0,x_1,\dots,x_t)=[x_0-(x_1+
\dots +x_t)]^2+x_1^2+
\dots +x_t^2.
\]
Therefore, $u_{\mathcal P}$ is positive definite.

Let $\cal P$ be a semichain \eqref{kjy}.
Suppose first that each $\mathcal P_i$ consists of two incomparable elements; that is, $\cal P$ is of the form
\begin{equation*}\label{ynp}
\{p_1,p_2\} \prec \{p_3,p_4\} \prec
\dots \prec\{p_{2s-1},p_{2s}\}
\end{equation*}
Then
\begin{align}\nonumber
u_{\mathcal P}(x_0,x_1,\dots,x_{2s})=&
[x_0-(x_1+ x_3+x_5+
\dots +x_{2s-1})\\&-(x_2+ x_4+x_6+
\dots +x_{2s})]^2 \label{dph}
\\&+(x_1-x_2)^2+
(x_3-x_4)^2+\dots+(x_{2s-1}-x_{2s})^2.
\nonumber
\end{align}
Therefore, $u_{\mathcal P}$ is nonnegative definite.
Suppose now that some of $\mathcal P_i$'s in \eqref{kjy} consist of one element.
Replacing $x_{2j}$ in \eqref{dph} by $0$ for each one-element $\mathcal P_j$, and renumbering the remaining $x_i$, we obtain $u_{\mathcal P}$, which is also nonnegative definite.

(ii)$\Leftarrow$(iii).
Suppose that $\cal P$ is not a chain. Then
it contains two incorporable elements $p_i$ and $p_j$. The form $u_{\mathcal P}(x_0,\dots,x_t)$ defined in \eqref{khtr} is equal to $0$ if $x_0=2$, $x_i=x_j=1$, and the other $x_1,\dots,x_t$ are zero. Therefore, $u_{\mathcal P}$ is not positive definite.

Suppose that $\cal P$ is not a semichain. Then
it contains three elements $p_i,p_j,p_k$  ($i<k$) such that $p_j$ is not comparable with $p_i$ and $p_k$; i.e., the Hasse diagram of $\{p_i,p_j,p_k\}$ is \eqref{lce}.
The form $u_{\mathcal P}(x_0,\dots,x_t)$ is equal to $-1$ (for $p_j\prec p_k$) or $-3$ (for $p_j\not\prec p_k$) if $x_0=3$, $x_i=x_j=x_k=1$, and the other $x_1,\dots,x_t$ are zero. Therefore, $u_{\mathcal P}$ is not nonnegative definite.

\section{Systems of subspaces of a vector space}\label{kudr}

By analogy with Definition \ref{def1}, \emph{$\cal P$-systems of subspaces of a vector space} are studied; that is, systems $(U_1,\dots,U_t)_U$ in which $U$ is  a vector space over a field $\mathbb F$ and $U_1,\dots,U_t$ are  its subspaces such that $U_i\subseteq U_j$ if $p_i\prec p_j$. Two systems $(U_1,\dots,U_t)_U$ and $(V_1,\dots,V_t)_V$ are \emph{isomorphic} if there exists a linear bijection $\varphi: U\to V$ such that $\varphi (U_1)=V_1,\dots, \varphi (U_t)=V_t$. Such systems are called \emph{filtered $\mathbb F$-linear representations of $\cal P$} in \cite[Chapter 3]{simson}.

Like \eqref{bbr}, each block matrix $A=[A_1|\dots|A_t]$ with $m$ rows defines the system of subspaces
$\mathcal U_A:=(U_1,\dots,U_t)_{\mathbb F^m}$ in which $U_j$ is the subspace of $\mathbb F^m$ spanned by columns of all $A_i$ such that $p_i\preccurlyeq p_j$. Thus, the theory of $\cal P$-systems of subspaces of a vector space reduces (see details in \cite[Chapter 3]{simson}) to the theory of matrix representations of posets, which was founded by Nazarova and Roiter \cite{NazarovaRoiter} and is presented in \cite{gab_roi,simson}.
A \emph{matrix representation of a poset
${\cal P}=\{p_1,\dots,p_t\}$} is a block matrix $A=[A_1|\dots|A_t]$.
Two matrix representations of $\cal P$ are \emph{isomorphic} if one can be obtained from the other by a sequence of transformations (i)--(iii) from Definition \ref{def3}(a) with ``elementary'' instead of ``unitary'' in (i). The \emph{direct sum} of representations is defined by \eqref{grp}.  A poset $\cal P$ is \emph{representation-finite} if it has only a finite number of nonisomorphic  indecomposable representations.
A poset $\cal P$ is \emph{wild} if the problem of classifying its representations up to isomorphism contains the problem of classifying matrix pairs up to simultaneous similarity $(M,N)\sim (S^{-1}MS,S^{-1}NS)$ (and hence, it contains the problem of classifying any system of linear operators and representations of any poset; see \cite{gel-pon_comm,bel-ser_compl}); the other posets  are \emph{tame}. The \emph{Tits form of $\cal P$} is the integral quadratic form
\begin{equation}\label{kde1}
q_{\mathcal P}(x_0,x_1,\dots,x_t):=x_0^2+x_1^2+\dots+x_t^2 + \sum_{p_i\prec p_j} x_ix_j - x_0(x_1+\dots+x_t)
\end{equation}
in which the sum is taken over all pairs $p_i,p_j\in\cal P$ satisfying $p_i\prec p_j$.

The nonunitary analog of Theorem \ref{ttt1} is the following classical result for representations of a finite poset $\cal P$ over an algebraically closed field $\mathbb F$ (that is, for $\cal P$-systems of subspaces of a vector space over $\mathbb F$):
\begin{itemize}
  \item[\rm(a)] $\mathcal P$ is representation-finite if and only if $\mathcal P$
does not contain a full poset whose Hasse diagram is one of the forms
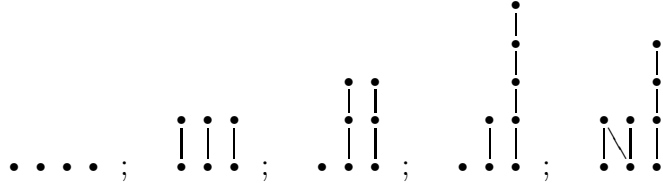
\[
\begin{split}
\xymatrix@C=4pt@R=7pt{
&&&&
&&&&&
&&&&&
&&&&*{\ci}&
&&&&
     \\
&&&&
&&&&&
&&&&&
&&&&*{\ci}\ar@{-}[u]&
&&&&*{\ci}
     \\
&&&&
&&&&&
&&&*{\ci}&*{\ci}&
&&&&*{\ci}\ar@{-}[u]&
&&&&*{\ci}\ar@{-}[u]
     \\
&&&&
&&*{\ci}&*{\ci}&*{\ci}&
&&&*{\ci}\ar@{-}[u]&*{\ci}\ar@{-}[u]&
&&&*{\ci}&*{\ci}\ar@{-}[u]&
&&*{\ci}&*{\ci}&*{\ci}\ar@{-}[u]
     \\
*{\ci}&*{\ci}&*{\ci}&*{\ci}&
;&&
*{\ci}\ar@{-}[u]&*{\ci}\ar@{-}[u]
&*{\ci}\ar@{-}[u]&
;&&
*{\ci}&*{\ci}\ar@{-}[u]
&*{\ci}\ar@{-}[u]&
;&&
*{\ci}&*{\ci}\ar@{-}[u]
&*{\ci}\ar@{-}[u]&
;&&
*{\ci}\ar@{-}[u]
&*{\ci}\ar@{-}[u]\ar@{-}[lu]
&*{\ci}\ar@{-}[u]
}
\end{split}
\]
$(a\prec b$ iff $a$ is under $b$ and
they are linked by a line$)$, if and only if the Tits form $q_{\mathcal P}$ is weakly positive definite $($see Remark 1 in Section \ref{s2}$)$.	

\item[\rm(b)] $\mathcal P$ is tame if and only if $\mathcal P$
does not contain a full poset whose Hasse diagram is one of the forms
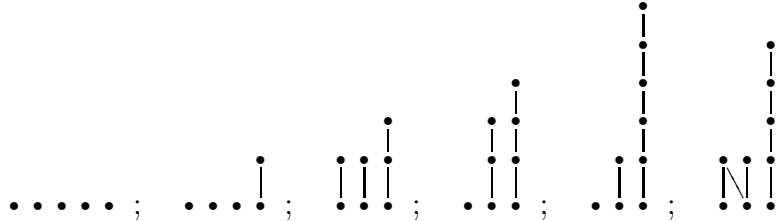
\[
\begin{split}
\xymatrix@C=3pt@R=7pt{
&&&&&
&&&&&&
&&&&&
&&&&&
&&&&*{\ci}&
&&&&
     \\
&&&&&
&&&&&&
&&&&&
&&&&&
&&&&*{\ci}\ar@{-}[u]&
&&&&*{\ci}
     \\
&&&&&
&&&&&&
&&&&&
&&&&*{\ci}&
&&&&*{\ci}\ar@{-}[u]&
&&&&*{\ci}\ar@{-}[u]
     \\
&&&&&
&&&&&&
&&&&*{\ci}&
&&&*{\ci}&*{\ci}\ar@{-}[u]&
&&&&*{\ci}\ar@{-}[u]&
&&&&*{\ci}\ar@{-}[u]
     \\
&&&&&
&&&&&*{\ci}&
&&*{\ci}&*{\ci}&*{\ci}\ar@{-}[u]&
&&&*{\ci}\ar@{-}[u]&*{\ci}\ar@{-}[u]&
&&&*{\ci}&*{\ci}\ar@{-}[u]&
&&*{\ci}&*{\ci}&*{\ci}\ar@{-}[u]
     \\
*{\ci}&*{\ci}&*{\ci}&*{\ci}&*{\ci}&
;&&
*{\ci}&*{\ci}&*{\ci}&*{\ci}\ar@{-}[u]&
;&&
*{\ci}\ar@{-}[u]&*{\ci}\ar@{-}[u]
&*{\ci}\ar@{-}[u]&
;&&
*{\ci}&*{\ci}\ar@{-}[u]
&*{\ci}\ar@{-}[u]&
;&&
*{\ci}&*{\ci}\ar@{-}[u]
&*{\ci}\ar@{-}[u]&
;&&
*{\ci}\ar@{-}[u]
&*{\ci}\ar@{-}[u]\ar@{-}[lu]
&*{\ci}\ar@{-}[u]
}
\end{split}
\] if and only if the Tits form $q_{\mathcal P}$ is weakly nonnegative definite.		
\end{itemize}
The proof of (a) and (b) can be found in \cite[Theorems 10.1, 15.3]{simson} and in \cite[Theorems 3.1.3, 3.1.4, 3.1.6]{haz}. The first equivalences in (a) and (b) were proved by Kleiner \cite{kle} and Nazarova \cite{naz}.

The conditions ``$q_{\mathcal P}$ is weakly positive definite'' and ``$q_{\mathcal P}$ is weakly nonnegative definite'' in (a) and (b) cannot be replaced by ``$q_{\mathcal P}$ is positive definite'' and ``$q_{\mathcal P}$ is nonnegative definite''.
For example, the Tits form $q(x_0,\dots,x_5)$ of
\begin{equation}\label{lche}
\begin{split}
\xymatrix@=15pt{
{p_5}\\
{p_3}\ar@{-}[u]&{p_4}\\
{p_1}\ar@{-}[u]\ar@{-}[ur]&{p_2}
\ar@{-}[u]\ar@{-}[ul]
}
\end{split}
\end{equation}
is weakly positive definite by (a), but it is not nonnegative definite since
$q(-1,2,2,-2,-2,-2)=-1$.  Bondarenko and Stepochkina \cite{bon-st} (see also \cite{bon-st1}) gave a list of posets  with
positive definite Tits form; it consists of four infinite series and 108 posets defined up to duality; this list was constructed in an alternative way in \cite{gas-sim}.

In contrast to this, we prove in the next theorem that ``positive definite'' and ``nonnegative definite'' can be replaced in Theorem \ref{ttt1} by ``weakly positive definite'' and ``weakly nonnegative definite''.

Denote by $\mathcal P\coprod \mathcal P$ the disjoint union of $\mathcal P$ with itself. We identify $\mathcal P\coprod \mathcal P$ with the poset $\mathcal P\cup \mathcal P'=\{p_1,\dots,p_t,p'_1,\dots,p'_t\}$ in with the elements of $\mathcal P$ are not comparable with the elements of $\mathcal P'$ and the order on $\mathcal P'$ is the same as on $\mathcal P$: $p'_i\prec p'_j$ if and only if $p_i\prec p_j$.

\begin{theorem}\label{th4}
{\rm(i)} The form \eqref{khtr} of $\mathcal P$ can be expressed via the Tits form of $\mathcal P \coprod \mathcal P$ as follows:
\begin{equation*}\label{kut}
u_{\mathcal P}(x_0,x_1,\dots,x_t)=q_{\mathcal P \coprod \mathcal P}(x_0,x_1,\dots,x_t,x_1,\dots,x_t).
\end{equation*}

{\rm(ii)} The form $u_{\mathcal P}$ is positive definite if and only if it is weakly positive definite.
The form $u_{\mathcal P}$ is nonnegative definite if and only if it is weakly nonnegative definite.
\end{theorem}

\begin{proof}
(i) This statement is obvious since the  Tits form \eqref{kde1} of $\mathcal P\coprod \mathcal P$ is
\begin{multline*}
	q_{\mathcal P \coprod\mathcal P}(x_0,x_1,\dots,x_t,x'_1,\dots,x'_t)\\=
	x_0^2+\sum_{i=1}^t (x_i^2 +  x_i^{\prime 2}) + \sum_{p_i\prec p_j}(x_ix_j+x'_ix'_j)
	- x_0\sum_{i=1}^t (x_i +x'_i).
\end{multline*}

(ii) This statement holds since all statements in Section \ref{sskw3} remains true if ``positive definite'' and ``nonnegative definite'' are replaced by ``weakly positive definite'' and ``weakly nonnegative definite''.
\end{proof}

\section{How the form $u_{\cal P}$ was constructed}\label{appen}

The form $u_{\cal P}$ defined in \eqref{khtr} is a unitary analog of the Tits forms for posets and quivers (see \eqref{kde1} and \cite[Section 2.4]{haz}). If $z$ is the dimension of any representation of a poset or quiver, then the Tits form is equal to the number of entries in the transforming matrices minus the number of entries in the matrices of the representation. The form $u_{\cal P}$ was constructed analogously:
{\it
let $A=[A_1|\dots|A_t]$ be a matrix representation  of $\cal P$ of size $m\times (n_1,\dots,n_t)$, let all entries of $A$ be independent parameters, and let $A$ be reduced by transformations $A\mapsto RAS$ of unitary $\cal P$-equivalence defined in \eqref{dru}. Then
\begin{equation}\label{heo}
\begin{split}
u_{\mathcal P}(m,n_1,\dots,n_t)=
\#[\text{\rm real parameters in $R$ and $S$}]\\-
\#[\text{\rm real parameters in $A$}]
\end{split}
\end{equation}}
(each complex parameter is counted as two real parameters).

Indeed,
by \eqref{khtr} $u_{\mathcal P}(m,n_1,\dots,n_t)$ is equal to
\begin{equation*}\label{khtrw}
m^2+2(n_1^2
+\dots+n_t^2) + 2\sum_{p_i\prec p_j} n_in_j - 2m(n_1+\dots+n_t)
\end{equation*}
in which
\begin{itemize}
\item $m^2$ is the number of real parameters of a general unitary $m\times m$ matrix $R$ (which gives unitary transformations of rows of $A$) since it is equal to the number of real parameters of a general Hermitian $m\times m$ matrix. They are equal due to Cayley's parametrization of a unitary matrix $U$ that does not have $-1$ as an eigenvalue:
\[U=(iI+H)(iI-H)^{-1}\] in which $H$ is a Hermitian matrix defined by
\[iH=(U+I)^{-1}(U-I).\] In particular, if $m=1$, then $U=[c]$ is given by one real parameter since  $c=e^{i\varphi}$, $0\le\varphi <2\pi$.

\item
$2(n_1^2+\dots+n_t^2)$ is the number of real parameters in the diagonal blocks $S_{11},S_{22},\dots,S_{tt}$ of $S$ (see \eqref{dru}); they give transformations of columns within blocks of $A$.
  \item
$2\sum_{p_i\prec p_j} n_in_j$ is the number of real parameters in nonzero off-diagonal blocks of $S$; they give additions of columns of one block to columns of another block.

  \item
$2m(n_1+\dots+n_t)$ is the number of real parameters in $A$,
\end{itemize}
which proves \eqref{heo}.

In reality, one real parameter in the reducing matrices $R$ and $S$ does not change $A$ since if  $c$ is a complex number with $|c|=1$, then $RAS=(cR)A(c^{-1}S)$. Therefore, the $2m(n_1+\dots+n_t)$ real parameters of $A$ are reduced by the $m^2+2(n_1^2
+\dots+n_t^2) + 2\sum_{p_i\prec p_j} n_in_j-1$ real parameters of $R$ and $S$, and so the number of real parameters of $A$ remaining after reduction is at least $1-u_{\mathcal P}(m,n_1,\dots,n_t)$. This leads to the hypothesis:
\begin{quote}\it
the number of real parameters of the set of block matrices of size $m\times (n_1,\dots,n_t)$ is at least $1-u_{\mathcal P}(m,n_1,\dots,n_t)$.
\end{quote}
Roughly speaking, this means that the classes of unitarily $\cal P$-equivalent block matrices of size $m\times (n_1,\dots,n_t)$ form a family depending on at least $1-u_{\mathcal P}(m,n_1,\dots,n_t)$ continuous real parameters.

A stronger statement for  representations of any quiver $Q$  over an algebraically closed field was proved by Kac \cite{kac}: if the set of indecomposable representations of dimension ${z}$ is nonempty, then
the number of its parameters is $1-q_Q({z})$, in which $q_Q$ is the Tits form of $Q$.
For unitary representations of a quiver (each vertex is assigned by a unitary space and each arrows is assigned by a linear mapping), the number of parameters of the set of indecomposable representations of a fixed dimension  was calculated in \cite[Section 3.3]{Serg_un}.

\section*{Acknowledgement}

V. Futorny is supported in part by the CNPq grant (301743/2007-0) and by the Fapesp grant (2010/50347-9). This work was done during the visit of K. Yusenko to the
University of S\~ao Paulo
as a postdoctoral fellow. He is grateful to the University of S\~ao Paulo for hospitality and to the Fapesp for financial support (2010/15781-0).

\end{document}